\documentclass{article}

\usepackage{fancyhdr}
\usepackage{tkz-graph}
\usepackage{tkz-berge}
\usepackage{circuitikz}
\usepackage{wrapfig}
\usepackage{caption}
\usetikzlibrary{arrows}
\usepackage[a4paper, left=4cm, right=4cm, top=3cm, bottom=3cm]{geometry}
\usepackage{microtype}
\usepackage[latin1]{inputenc}
	
\usepackage{chngcntr} %For figure counter inside section
\usepackage{enumitem} %einfacheres formatieren von itemize/enumerate

\usepackage{float}
\usepackage{tikz}
\usepackage{pgfplots} %zeichnen
\usepackage{graphicx}

\usepackage{dsfont}
\usepackage{amssymb}
\usepackage{amsmath}
\usepackage{amsthm} 
\usepackage{ifthen}

\usepackage[pdfborder={0 0 0}]{hyperref}

\allowdisplaybreaks[1] %allows pagebreak in align

\renewcommand{\bar}{\overline}

\newcommand{\set}[1]{\left\lbrace #1 \right\rbrace}
\newcommand{\setdef}[2]{\left\lbrace #1  ~ | ~  #2 \right\rbrace}

\newcommand{\all}{\ \forall\ }
\newcommand{\ex}{\ \exists\ }

\newcommand{\restr}{\mathord\restriction}

\def\N{\mathbb N}
\def\Z{\mathbb Z}
\def\R{\mathbb R}
\def\E{\mathbb{E}}
\def\P{\mathbb{P}}

\def\lb{\left(}
\def\rb{\right)}

\def\1{\mathds{1}}

\newcommand{\bP}{\bar{\P^n_x}}

\makeatother
\def\email#1{e-mail: #1}

\def\keywords#1{\par\medskip
\noindent\textbf{Keywords.} #1}

\def\amsclass#1{\par\medskip
	\noindent\textbf{AMS Subject Classification (2010):} #1}

\begin{document}
	\numberwithin{equation}{section}	
	%\numberwithin{figure}{section}	
	
	\theoremstyle{definition}
	\newtheorem{dfn}{Definition}[section]
	\newtheorem{rem}[dfn]{Remark}
	\newtheorem{example}[dfn]{Example}
	
	\theoremstyle{plain} 
	\newtheorem{lemma}[dfn]{Lemma}
	\newtheorem{prop}[dfn]{Proposition}
	\newtheorem{cor}[dfn]{Corollary}
	\newtheorem{thm}[dfn]{Theorem}
	
\title{On the Probabilistic Representation of the Free Effective Resistance of Infinite Graphs}
\author{%
	Tobias Weihrauch%
	\footnote{T. Weihrauch: Universit\"{a}t Leipzig, Fakult\"{a}t f\"{u}r Mathematik und Informatik, Augustusplatz 10, 04109 Leipzig, Germany; \email{weihrauch@math.uni-leipzig.de}
	\vspace{0.2cm}}
	\qquad %
	Stefan Bachmann%
	\footnote{S. Bachmann: Universit\"{a}t Leipzig, Fakult\"{a}t f\"{u}r Mathematik und Informatik, Augustusplatz 10, 04109 Leipzig, Germany; \email{bachmann@math.uni-leipzig.de}
	\vspace{0.2cm}}
	\\Universit\"{a}t Leipzig}
\maketitle

\begin{abstract}
	We completely characterize when the free effective resistance of an infinite graph can be expressed in terms of simple hitting probabilities of the graphs random walk.
	\amsclass{Primary 05C63 %Infinite graphs
		$\cdot$ 05C81; %Random walks on graphs
		Secondary 60J10 %Markov chains (discrete-time Markov processes on discrete state spaces)
		$\cdot$ 05C12 %Distance in graphs
	}
\end{abstract}

\section{Introduction}
We consider undirected, connected graphs with no multiple edges and no self-loops. Each edge $(x,y)$ is given a positive weight $c(x,y)$. A possible interpretation is that $(x,y)$ is a resistor with resistance $1/c(x,y)$. The graph then becomes an electrical network.

More precisely, a \emph{graph} $G = (V,c)$ consists of an at most countable \emph{set of vertices} $V$ and a \emph{weight function} $c : V \times V \to \R_{\geq 0}$ such that $c$ is symmetric and for all $x \in V$, we have $c(x,x) = 0$ and
\begin{equation*}
c_x := \sum_{y \in V} c(x,y) < \infty.
\end{equation*}
We think of two vertices $x,y$ as being \emph{adjacent} if $c(x,y) > 0$.

For $x \in V$, let $\P_x$ be the random walk on $G$ starting in $x$. It is the Markov chain defined by the transition matrix
\begin{equation*}
p(x,y) = \frac{c(x,y)}{c_x}~, x,y \in V
\end{equation*}
and initial distribution $\delta_x$. We will think of $\P_x$ as a probability measure on $\Omega = V^{\N_0}$ equipped with the $\sigma$-algebra $(2^V)^{\otimes \N_0}$. If not explicitly stated otherwise, we will from now on assume that every occurring graph is connected. In that case, $\P_x$ is irreducible.

For a set of vertices $A \subseteq V$, let
\begin{align*}
\tau_A &:= \inf\setdef{k \geq 0}{\omega_k \in A} \text{ and}\\
\tau_A^+ &:= \inf\setdef{k \geq 1}{\omega_k \in A}
\end{align*}
be \emph{hitting times of $A$}. For $x \in V$, we use the shorthand notation $\tau_{\set{x}} =: \tau_x$.

%The \emph{ball of radius $n$ around $x$} is denoted by $B_n(x)$ and defined inductively. Let $B_1(x) = N(x) \cup \set{x}$ and for $n \geq 2$, let
%\begin{equation}
%B_n(x) := \bigcup_{y \in B_{n-1}(x)} N(y)
%\end{equation}

Suppose that $G$ is finite. Ohm's Law states that the \emph{effective resistance} $R(x,y)$ between to vertices $x,y$ is the voltage drop needed to induce an electrical current of exactly 1 ampere from $x$ to $y$. 

The relationship between electrical currents and the random walk of $G$ has been studied intensively \cite{DoyleSnell1984, JorgensenPearse2009, LyonsPeres2015,Tetali1991}. For finite graphs, $x \neq y$, one has the following probabilistic representations
\begin{align}
\label{eq:EffResProbReprFiniteExpectation}
R(x,y)& = \frac{1}{c_x} \E_x\left[ \sum_{k=0}^{\tau_y-1} \mathds{1}_x \right]\\
\label{eq:EffResProbReprFiniteLessOrEqual}
& = \frac{1}{c_x \cdot \P_x[\tau_y \leq \tau_x^+]}\\
\label{eq:EffResProbReprFiniteStrictlyLess}
& = \frac{1}{c_x \cdot \P_x[\tau_y < \tau_x^+]}~.
\end{align}
Note that $(c_z)_{z \in V}$ is an invariant measure of $p$. A proof of the first equality in the unweighted case can be found in \cite{Tetali1991} and can be extended to fit our more general context. 
To see that (\ref{eq:EffResProbReprFiniteExpectation}) equals (\ref{eq:EffResProbReprFiniteLessOrEqual}), realize that $\sum_{k=0}^{\tau_y-1} \mathds{1}_x$ is geometrically distributed with parameter $\P_x[\tau_x^+ < \tau_y]$. For the last equality, use that any finite graph is recurrent and thus $\P_x[\tau_x^+ =\tau_y = \infty] = 0$.

The subject of effective resistances gets much more complicated on infinite graphs since those may admit multiple different notions of effective resistances. Recurrent graphs, however, have a property which is often referred to as \emph{unique currents} \cite{LyonsPeres2015} and consequently also have one unique effective resistance. In this case, the above representation holds \cite{Barlow2017,Weihrauch2018}. Indeed, \cite[Theorem 2.61]{Barlow2017} states the more general inequalities
\begin{equation}
\label{eq:BarlowBounds}
\frac{1}{c_x \cdot \P_x[\tau_y \leq \tau_x^+]} \leq R^F(x,y) \leq \frac{1}{c_x \cdot \P_x[\tau_x < \tau_x^+]}
\end{equation}
for the free effective resistance $R^F$ (see Section \ref{section:FreeEffRes}) of any infinite graph. 

In \cite[Corollary 3.13 and 3.15]{JorgensenPearse2009}, it is suggested that one has 
\begin{equation}
\label{eq:FreeEffResProbRepr}
R^F(x,y) = \frac{1}{c_x \cdot \P_x[\tau_y < \tau_x^+]}
\end{equation}
on all transient networks. However, this is false as our example in Section \ref{section:InfiniteT} shows.

The main result of this work (Corollary \ref{cor:MainResult}) states that the free effective resistance of a transient graph $G = (V,c)$ admits the representation (\ref{eq:FreeEffResProbRepr}) for all $x,y \in V$ if and only if $G$ is a subgraph of an infinite line. Corollary \ref{cor:LowerBoundCharacterization} states that the lower bound in (\ref{eq:BarlowBounds}) are attained if and only if $G$ is recurrent. %TODO in particular, we improve barlows statement

\section{Free effective resistance}
\label{section:FreeEffRes}
Let $G = (V,c)$ be an infinite connected graph. For any $W \subseteq V$, let $G \restr_W := (W, c \restr_W)$ be the \emph{subgraph of $G$ induced by $W$}.
We say a sequence $(V_n)_{n \in\N}$ of subsets of $V$ is a \emph{finite exhaustion of $V$} if $|V_n|< \infty$, $V_n \subseteq V_{n+1}$ and $V = \cup_{n \in\N} V_n$. Define $G_n = (V_n, c_n):= G\restr_{V_n}$.

\begin{dfn}
	Let $(V_n)_{n \in \N}$ be any finite exhaustion of $V$ such that $G_n$ is connected. For $x,y \in V$, the \emph{free effective resistance} $R^F(x,y)$ of $G$ is defined by
	\begin{equation*}
	R^F(x,y) = \lim_{n \to \infty} R_{G_n}(x,y).
	\end{equation*}
\end{dfn}
\begin{rem}
	The fact that $R_{G_n}(x,y)$ converges is due to Rayleigh's monotonicity principle (see e.g. \cite{BLPS2001, Grimmett2010}).
\end{rem}

We denote by $\P^n_x$ the random walk on $G_n$ starting in $x$ with transition matrix $p_n$. Since we can extend it to a function on $V$ by defining $p_n(x,y) = 0$ whenever $x \notin V_n$ or $y \notin V_n$, $\P^n_x$ is a probability measure on $\Omega = V^{\N_0}$ and we have
\begin{equation*}
p_n(x,y) = \frac{c_n(x,y)}{(c_n)_x} = \frac{c(x,y)}{\sum_{w \in V_n} c(x,w)}
\end{equation*}
for all $x,y \in V_n$.

\begin{rem}
	Note that 
	\begin{equation*}
		p_n(x,y) = p(x,y) \cdot \frac{c_x}{(c_n)_x} = p(x,y) \cdot \lb 1 + \frac{\sum_{v \notin V_n} c(x,v)}{\sum_{v \in V_n} c(x,v)}\rb \geq p(x,y).
	\end{equation*}
\end{rem}

Since
\[
R_{G_n}(x,y) = \frac{1}{(c_n)_x \P^n_x[\tau_y < \tau_x^+]}
\]
for all $n \in \N$ and $c_x = \lim_{n \to \infty} (c_n)_x$, (\ref{eq:FreeEffResProbRepr}) holds if and only if
\begin{equation}
\label{eq:ConvergenceOfProbabilites}
\lim_{n \to \infty} \P^n_x[\tau_y < \tau_x^+] = \P_x[\tau_y < \tau_x^+].
\end{equation}
Analogously, the lower bound of (\ref{eq:BarlowBounds}) is attained if and only if
\begin{equation}
\label{eq:ConvergenceOfProbabilitesLowerBound}
\lim_{n \to \infty} \P^n_x[\tau_y < \tau_x^+] = \P_x[\tau_y \leq \tau_x^+].
\end{equation}

\section{The transient $\mathcal{T}$}
\label{section:InfiniteT}
We will now show that (\ref{eq:FreeEffResProbRepr}) does not hold in general. Consider the graph $\mathcal{T}$ shown in Figure \ref{fig:TGraph}. It is transient and we have $R^F(B,T) = 2$. However, 
\begin{align*}
\P_B[\tau_T < \tau_B^+] & = \P_0[\tau_T < \tau_B]\\
& = 1 - \P_0[\tau_B \leq \tau_T]\\
& = 1 - \P_0[\tau_B < \tau_T] - \P_0[\tau_B = \tau_T = \infty].
\end{align*}
Due to the symmetry of $\mathcal{T}$ we have $ \P_0[\tau_B < \tau_T] =  \P_0[\tau_T < \tau_B]$. Together with the transience of $\mathcal{T}$, this implies
\[
\P_B[\tau_T < \tau_B^+] = \P_0[\tau_T < \tau_B] = \frac{1 -  \P_0[\tau_B = \tau_T= \infty]}{2} < \frac{1}{2}
\]
and
\[
\P_B[\tau_T \leq \tau_B^+] = \P_0[\tau_T \leq \tau_B] = \frac{1 +  \P_0[\tau_B = \tau_T= \infty]}{2} > \frac{1}{2}.
\]
More precisely, one can compute 
$$\P_B[\tau_T < \tau_B^+] = \frac{2}{5} \text{ and } \P_B[\tau_T \leq \tau_B^+] = \frac{3}{5}~.$$
Hence,
\[
\frac{1}{c_B \P_B[\tau_T < \tau_B^+]} \neq R^F(B,T)
\]
and
$$\frac{1}{c_B}\E_B\left[\sum_{k=0}^{\tau_T - 1} \mathds{1}_B(\omega_k) \right] = \frac{1}{c_B\P_B[\tau_T \leq \tau_B^+]} \neq R^F(B,T).$$
\begin{figure}
	\centering
	\begin{tikzpicture}[-,auto, node distance = 1.2cm, every loop/.style={}, vertex/.style={draw, circle, fill=black!10, inner sep=0mm, minimum size=6mm}]
	\node[vertex] (B1) {$B$};
	\node[vertex, above of= B1] (A) {$0$};
	\node[vertex, above of= A] (T) {$T$};
	\node[vertex, right of =A] (B) {$1$};
	\node[vertex, right of=B] (C) {$2$};
	\node[vertex, right of=C] (D) {$3$};
	\node[vertex, right of=D] (E) {$4$};
	\node[right of=E] (F) {$\ldots$};
	
	\path 	(B1) edge node {$1$} (A)
	(T) edge node[left] {$1$} (A)
	(A) edge node {$1$} (B)
	(B) edge node {$2$} (C)
	(C) edge node {$4$} (D)
	(D) edge node {$8$} (E)
	(E) edge node {$16$} (F);
	\end{tikzpicture}
	\caption{\label{fig:TGraph}The transient graph $\mathcal{T}$}
\end{figure}
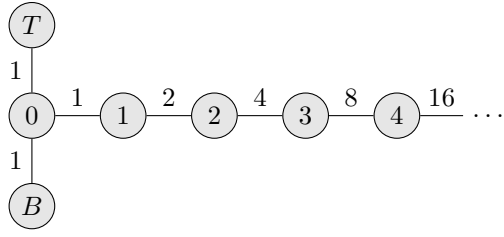
\begin{rem}
	Note that, although $\mathcal{T}$ is transient, it has unique currents since every harmonic function is constant. This shows that whether (\ref{eq:FreeEffResProbRepr}) holds is more tightly connected to the transience of the graph than to the uniqueness of currents (which is equivalent to the existence of harmonic Dirichlet functions \cite{LyonsPeres2015}).
\end{rem}

\section{Probability of paths}

To check whether (\ref{eq:ConvergenceOfProbabilites}) holds, it is useful to write both sides as sums of probabilities of paths.

A sequence $\gamma = (\gamma_0, \ldots, \gamma_n) \in V^{n+1}$ is called a \emph{path (of length n)} in $G$ if $c(\gamma_k, \gamma_{k+1}) > 0$ for all $k = 0, \ldots, n-1$. We denote by $L(\gamma)$ the length of $\gamma$ and by $\Gamma_G$ the set of all paths in $G$. A path $\gamma$ is called \emph{simple} if it does not contain any vertex twice. The probability of $\gamma$ with respect to $\P_x$ is defined by
\begin{equation*}
\P_x(\gamma) := \P_x(\set{\gamma} \times V^{\N}) = \mathds{1}_x(\gamma_0) \cdot \prod_{k=0}^{L(\gamma)-1} p(\gamma_k, \gamma_{k+1})~.
\end{equation*}

We say $\gamma$ is $x \to y$ if $\gamma_0 = x, \gamma_{L(\gamma)} = y$ and $\gamma_k \notin \{x,y\}$ for all $k = 1,\ldots, L(\gamma) - 1$. We denote by $\Gamma_G(x,y)$ the set of all paths $x \to y$ in $G$.

For $A \subseteq V$, let
\[
\Gamma_G(x,y;A) := \setdef{\gamma \in \Gamma_G(x,y)}{\gamma_k \in A \text{ for all } k = 0, \ldots, L(\gamma)}
\]
be the set of all paths $x \to y$ in $G$ which only use vertices in $A$. 

Using this notion and $\Gamma_{G_n}(x,y) = \Gamma_G(x,y;V_n)$,  (\ref{eq:ConvergenceOfProbabilites}) becomes
\begin{equation}
\label{eq:ConvergenceSumsOfPaths}
\lim_{n \to \infty} \sum_{\gamma \in \Gamma_G(x,y;V_n)} \P^n_x(\gamma) = \sum_{\gamma \in \Gamma_G(x,y)} \P_x(\gamma)
\end{equation}
Since $\Gamma_G(x,y;V_n)$ increases to $\Gamma_G(x,y)$, this looks like an easy application of either the Monotone Convergence Theorem or the Dominated Convergence Theorem.  However, both are not applicable since $\P^n_x(\gamma)$ may be strictly greater than $\P_x(\gamma)$.

To investigate when exactly (\ref{eq:ConvergenceSumsOfPaths}) holds, we will introduce another random walk on $V$ which can be considered an intermediary between $\P^n_x$ and $\P_x$.

\section{Extended finite random walk}

The only difference in the behavior of $\P_x$ and $\P^n_x$ occurs when $\P_x$ leaves $V_n$. Instead, $\P^n_x$ is basically reflected back to a vertex in $V_n$. We will now construct an intermediary random walk which still has a finite state space, models the behavior of stepping out of $V_n$ and has the same transition probabilities as $\P_x$ in $V_n$. This is done by adding \emph{boundary vertices} to $G_n$ wherever there is an edge from $V_n$ to $V \setminus V_n$. 

For any set $A \subseteq V$, let
\[
\partial_i A := \setdef{v \in A}{\ex w \in V \setminus A : c(v,w) > 0}
\]
be the \emph{inner boundary} and $\partial_o A := \partial_i (V \setminus A)$ be the \emph{outer boundary} of $A$ in $G$.

For any $v \in \partial_i A$, let $\bar{v}$ be a copy of $v$. Define $\bar{G_n} = (\bar{V_n}, \bar{c_n})$ where
\[
\bar{V_n} = V_n \cup \setdef{\bar{v}}{v \in \partial_i V_n},
\]
and $\bar{c_n}$ is defined as follows. For $x,y \in \bar{V_n}$, let 
\begin{equation*}
\bar{c_n}(x,y) = \bar{c_n}(y,x) = 
\begin{cases}
c(x,y) &, ~ x,y \in V_n\\
\sum_{z \notin V_n} c(x,z) & , ~ y = \bar{x}\\
0 & , \text{ otherwise}
\end{cases}.
\end{equation*}
In particular, we have $(\bar{c_n})_x = c_x$ for all $x \in V$. We denote by $\bar{\P^n_x}$ the random walk on $\bar{G_n}$ starting in $x$ with transition matrix $\bar{p_n}$ given by
\[
\bar{p_n}(x,y) = \frac{\bar{c_n}(x,y)}{(\bar{c_n})_x}~.
\]
Furthermore, let $V_n^* := \bar{V_n} \setminus V_n$.

\begin{example}
	Let $G$ be the lattice $\Z^2$ with unit weights, see Figure \ref{fig:ExampleZ2}. Furthermore, let $V_n := \set{-n\ldots,0,\ldots,n}^2$.
	\begin{figure}
		\centering
		\begin{tikzpicture}[
		vertex/.style={draw, circle, fill=black, minimum size=1.5mm, inner sep=0, align=center, scale=0.75},
		faded/.style={draw=black!50, dashed}
		]
		%main Z^2 grid -2...2
		\foreach \x in {-2,...,2}
		\foreach \y in {-2,...,2}
		\node [vertex] (\x\y) at (\x, \y) {}; 
		
		%grid lines
		\foreach \x in {-2,...,2}
		\foreach \y[count=\yi] in {-2,...,1}
		\pgfmathtruncatemacro{\i}{\yi-2} \draw (\x\y) -- (\x\i) (\y\x) -- (\i\x);
		
		%outside stuff
		\foreach \t in {-2,...,2}
		{
			\node (t\t) at (\t, 3) {};
			\draw [faded] (t\t) -- (\t 2);
			\node (b\t) at (\t, -3) {};
			\draw [faded] (b\t) -- (\t -2);
			\node (l\t) at (-3, \t) {};
			\draw [faded] (l\t) -- (-2\t);
			\node (r\t) at (3, \t) {};
			\draw [faded] (r\t) -- (2\t);
		}
		\end{tikzpicture}
		\caption{\label{fig:ExampleZ2} The lattice $\Z^2$.}
	\end{figure}
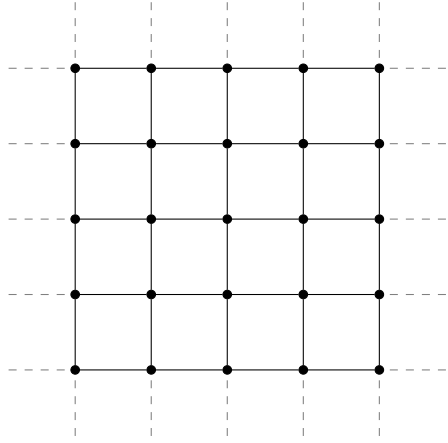
$G_1$ and $\bar{G_1}$ can be seen in Figure \ref{fig:ExampleExtension}. Note that $c((1,1), \bar{(1,1)}) = 2$ since $(1,1)$ has two edges leaving $V_1$ in $G$.
	\begin{figure}[h]
			\centering
			\begin{tikzpicture}[
			vertex/.style={draw, circle, fill=black, minimum size=1.5mm, inner sep=0, align=center, scale=0.75},
			faded/.style={draw=black!50, dashed}
			]
			%main Z^2 grid -1...1
			\foreach \x in {-1,...,1}
			\foreach \y in {-1,...,1}
			\node [vertex] (\x\y) at (\x, \y) {}; 
			
			%grid lines
			\foreach \x in {-1,...,1}
			\foreach \y[count=\yi] in {-1,0}
			\pgfmathtruncatemacro{\i}{\yi-1} \draw (\x\y) -- (\x\i) (\y\x) -- (\i\x);
			
			\draw [->] (2.5,0) -- (3.5,0);
			
			% RIGHT SIDE
			\def\offsetX{7}
			%main Z^2 grid -1...1
			\foreach \x in {-1,...,1}
			\foreach \y in {-1,...,1}
			\node [vertex] (R\x\y) at (\offsetX+\x, \y) {}; 
			
			%grid lines
			\foreach \x in {-1,...,1}
			\foreach \y[count=\yi] in {-1,0}
			\pgfmathtruncatemacro{\i}{\yi-1} \draw (R\x\y) -- (R\x\i) (R\y\x) -- (R\i\x);

			%boundary vertices (center)
			\foreach \t in {-1,1}
			{
				\node [vertex] (vo\t) at (\offsetX, \t*2) {};
				\draw (vo\t) -- (R0\t);
				
				\node [vertex] (ho\t) at (2*\t+\offsetX, 0) {};
				\draw (ho\t) -- (R\t 0);
			}
		
			%boundary vertices (corners)
			\foreach \t in {-1,1}
			{
				\node [vertex] (o1\t\t) at (\offsetX+\t*2, \t*2) {};
				\draw (o1\t\t) -- (R\t\t) node [midway,above,yshift=1mm] {2};
				
				\pgfmathtruncatemacro{\s}{-\t}
				\node [vertex] (o\t\s) at (\offsetX+\t*2, \s*2) {};
				\draw (o\t\s) -- (R\t\s) node [midway,above,yshift=1mm] {2};
			}
			\end{tikzpicture}
			\caption{\label{fig:ExampleExtension} $G_1$ (left) and $\bar{G_1}$ (right) for $G = \Z^2$.}
	\end{figure}
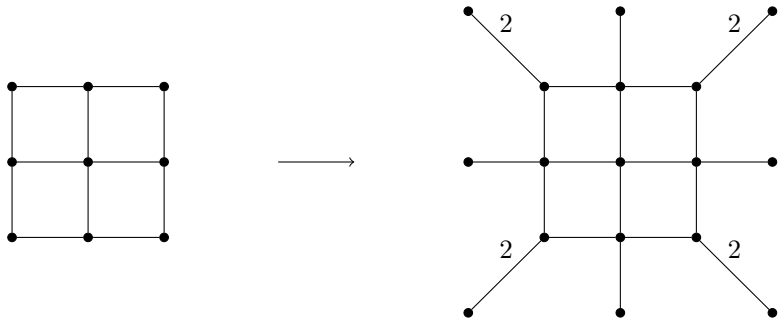
\end{example}

\begin{lemma}[Relation of $p_n, \bar{p_n}$ and $p$]
	\label{lemma:RelationOfRandomWalks}
	For $x,y \in V_n$ we have
	\begin{equation*}
	p_n(x,y) \geq p(x,y) = \bar{p_n}(x,y).
	\end{equation*}
	For $x,y \in V$ and $m \in \N$ such that $x,y \in V_m$, we have
	\begin{equation*}
	\lim_{n \to \infty} p_n(x,y) = p(x,y) = \bar{p_m}(x,y).
	\end{equation*}
\end{lemma}
Note that for $n \in \N$, we have
\[
\Gamma_{G_n}(x,y) = \Gamma_{\bar{G_n}}(x,y;V_n) = \Gamma_G(x,y;V_n).
\]
By Lemma \ref{lemma:RelationOfRandomWalks}, the following holds for all $x,y \in V_n$.
\begin{equation*}
\all m \geq n \all \gamma \in \Gamma_G(x,y;V_m): \P_x(\gamma) = \bar{\P^m_x}(\gamma)
\end{equation*}
The connection between $\P_x^n(\gamma)$ and $\bP(\gamma)$ is a bit more intricate.
\begin{dfn}
	For $x,y \in V_n$, let $\pi : \Gamma_{\bar{G_n}}(x,y) \to \Gamma_{G_n}(x,y)$ be the projection of $\Gamma_{\bar{G_n}}(x,y)$ onto $\Gamma_{G_n}(x,y)$ which removes all steps of the form $(\bar{v},v)$. 
	
	More precisely, for $\gamma \in \Gamma_{G_n}(x,y)$, let $\pi(\gamma) = \gamma$. For general $\gamma \in \Gamma_{\bar{G_n}}(x,y)$, define $\pi(\gamma)$ inductively over $L(\gamma)$. If $L(\gamma) \leq 2$, then $\gamma \in \Gamma_{G_n}(x,y)$ and thus $\pi(\gamma) =\gamma$. For $n := L(\gamma) > 2$, let $\gamma = (x, \gamma_1, \ldots, \gamma_{n-1}, y)$ and $\gamma' = (x, \gamma_1, \ldots, \gamma_{n-3})$. Now, $(\gamma_{n-2}, \gamma_{n-1}, y)$ can either be of the form $(v,w,y)$ with $v,w \in V_n$ or $(\bar{v},v,y)$ with $v \in V_n$. In the former case, let
	\[
	\pi(\gamma) := (\pi(\gamma'), \gamma_{n-2}, \gamma_{n-1}, y),
	\]
	in the latter
	\[
	\pi(\gamma) := (\pi(\gamma'),y).
	\]
\end{dfn}
\begin{lemma}
	\label{lemma:PathProbabilityFiniteAndExtendedRandomWalk}
	For all $x \neq y$ and $\gamma \in \Gamma_{G_n}(x,y)$, we have
	\begin{equation*}
	\P_x^n(\gamma) = \frac{c_x}{(c_n)_x} \cdot \sum_{\bar{\gamma} \in \pi^{-1}(\gamma)} \bar{\P_x^n}(\bar{\gamma}).
	\end{equation*}
\end{lemma}
\begin{proof}
	We prove the claim via induction over $L(\gamma)$.
	For $L(\gamma) = 1$, we have $\gamma = (x,y)$. By definition of $\pi$, any preimage $\gamma'\in \pi^{-1}(\gamma)$ is in $\Gamma_{\bar{G_n}}(x,y)$ and thus visits $x$ and $y$ exactly once. Hence, $\pi^{-1}(\gamma) = \set{\gamma}$ and
	\[
	\frac{c_x}{(c_n)_x} \cdot \sum_{\bar{\gamma} \in \pi^{-1}(\gamma)} \bar{\P_x^n}(\bar{\gamma}) = \frac{c_x}{(c_n)_x} \cdot\bar{\P_x^n}(\gamma) = \frac{c_x}{(c_n)_x} \cdot \frac{c(x,y)}{c_x} = p_n(x,y) = \P_x^n(\gamma).
	\]
	Suppose that $n := L(\gamma) > 1$ and let $\gamma' = (\gamma_0, \ldots, \gamma_{n-1})$ and $v = \gamma_{n-1}$ Then, 
	\[
	\P_x^n(\gamma) = \P_x^n(\gamma') \cdot p_n(v, y)
	\]
	and
	\[
	\pi^{-1}(\gamma) = \setdef{(\bar{\gamma}',\underbrace{(\bar{v}, v), \ldots,(\bar{v}, v)}_{k~\text{times}},y)}{\bar{\gamma}' \in \pi^{-1}(\gamma'), k \in \N_0}.
	\]
	Since
	\[
	\sum_{k=0}^{\infty} (\bar{p_n}(v,\bar{v})\cdot \underbrace{\bar{p_n}(\bar{v},v)}_{=1})^k = \frac{1}{1 - \bar{p_n}(v, \bar{v})} = \frac{c_v}{c_v - \sum_{w \notin V_n} c(v,w)} = \frac{c_v}{(c_n)_v}~,
	\]
	we have 
	\[
	p_n(v,y) = \frac{c_v}{(c_n)_v} \cdot \bar{p_n}(v,y) = \left( \sum_{k=0}^{\infty} (\bar{p_n}(v,\bar{v})\cdot \underbrace{\bar{p_n}(\bar{v},v)}_{=1})^k \right) \cdot \bar{p_n}(v,y).
	\]
	Hence,
	\begin{align*}
	\P_x^n(\gamma) & = \P_x^n(\gamma') \cdot p_n(v, y) \\
	& = \frac{c_x}{(c_n)_x}\left(\sum_{\bar{\gamma}' \in \pi^{-1}(\gamma')} \bP(\bar{\gamma}') \right) \cdot \left( \sum_{k=0}^{\infty} (\bar{p_n}(v,\bar{v})\cdot \bar{p_n}(\bar{v},v))^k \right) \cdot \bar{p_n}(v,y)\\
	& = \frac{c_x}{(c_n)_x} \sum_{\bar{\gamma}' \in \pi^{-1}(\gamma')} \sum_{k=0}^{\infty} \left[ \bP(\bar{\gamma}') \cdot (\bar{p_n}(v,\bar{v})\cdot \bar{p_n}(\bar{v},v))^k \cdot \bar{p_n}(v,y)\right]\\
	& = \frac{c_x}{(c_n)_x}\sum_{\bar{\gamma}' \in \pi^{-1}(\gamma')} \sum_{k=0}^{\infty} \bP((\bar{\gamma}',\underbrace{(\bar{v}, v),\ldots,(\bar{v}, v)}_{k~\text{times}},y))\\
	& = \frac{c_x}{(c_n)_x}\sum_{\bar{\gamma} \in \pi^{-1}(\gamma)} \bar{\P_x^n}(\bar{\gamma}).
	\end{align*}
\end{proof}

\begin{prop}
	For $x,y \in V_n$, we have
	\begin{equation*}
	\bP[\tau_y < \tau_x^+] = \frac{(c_n)_x}{c_x} \cdot \P^n_x[\tau_y < \tau_x^+].
	\end{equation*}
\end{prop}
\begin{proof}
Using
\[
\Gamma_{\bar{G_n}}(x,y) = \bigsqcup_{\gamma \in \Gamma_{G_n}(x,y)} \pi^{-1}(\gamma),
\]
we compute
\begin{align*}
\P_x^n[\tau_y < \tau_x^+] & = \sum_{\gamma \in \Gamma_{G_n}(x,y)} \P_x^n(\gamma)  = \sum_{\gamma \in \Gamma_{G_n}(x,y)} \left( \frac{c_x}{(c_n)_x} \cdot \sum_{\bar{\gamma} \in \pi^{-1}(\gamma)} \bar{\P_x^n}(\bar{\gamma}) \right)\\
& = \frac{c_x}{(c_n)_x} \cdot  \sum_{\bar{\gamma} \in \Gamma_{\bar{G_n}}(x,y)} \bar{\P_x^n}(\bar{\gamma}) = \frac{c_x}{(c_n)_x} \cdot \bar{\P_x^n}[\tau_y < \tau_x^+].
\end{align*}
\end{proof}
Since we now have clarified the relation between $\P^n_x$, $\bP$ and $\P_x$, we can return our attention to (\ref{eq:ConvergenceOfProbabilites}).
\begin{prop}
	\label{prop:CharacterizationConvergence}
	For $x,y \in V$, $x \neq y$, we have 
	\begin{equation*}
	\lim_{n \to \infty} \P^n_x[\tau_y < \tau_x^+] = \P_x[\tau_y < \tau_x^+]
	\end{equation*}
	if and only if 
	\begin{equation}
	\label{eq:HittingTimeBoundaryLimitZero}
	\lim_{n \to \infty} \bar{\P^n_x}[\tau_{V_n^*} < \tau_y < \tau_x^+] = 0.
	\end{equation}
\end{prop}
\begin{proof}
	We have
	\begin{align}
	\label{eq:EscapeProbInfiniteRWInTermsOfExtendedRW}
	\P_x[\tau_y <\tau_x^+] & = \sum_{\gamma \in \Gamma_G(x,y)} \P_x(\gamma) = \lim_{n \to \infty} \sum_{\gamma \in \Gamma_G(x,y;V_n)} \P_x(\gamma) \\
	\notag & = \lim_{n \to \infty} \sum_{\gamma \in \Gamma_G(x,y;V_n)} \bar{\P^n_x}(\gamma)\notag
	\end{align}
	and
	\begin{equation}
		\label{eq:EscapeProbFiniteRWInTermsOfExtendedRW}
		\P^n_x[\tau_y < \tau_x^+] = \frac{c_x}{(c_n)_x} \cdot \bar{\P^n_x}[\tau_y < \tau_x^+] = \frac{c_x}{(c_n)_x} \cdot  \sum_{\gamma \in \Gamma_{\bar{G_n}}(x,y)} \bar{\P^n_x}(\gamma).
	\end{equation}
	Since $\Gamma_G(x,y;V_n) = \Gamma_{\bar{G_n}}(x,y;V_n)$ and $(c_n)_x \to c_x$, it follows that $\P^n_x[\tau_y < \tau_x] \to \P_x[\tau_y < \tau_x]$ holds if and only if
	\[
	\lim_{n \to \infty} \sum_{
		\substack{
			\gamma \in \Gamma_{\bar{G_n}}(x,y)\\ 
			\gamma \notin \Gamma_{\bar{G_n}}(x,y;V_n)
		}} \bar{\P^n_x}(\gamma)\ = 0.
	\]
	This is the same as
	\[
	\lim_{n \to \infty} \bar{\P^n_x}[\tau_{V_n^*} < \tau_y < \tau_x^+] = 0.
	\]
\end{proof}
Using the same approach, we can also characterize when (\ref{eq:ConvergenceOfProbabilitesLowerBound}) holds.
\begin{prop}
	\label{prop:CharacterizationConvergenceLowerBound}
	For $x,y \in V$, $x \neq y$, we have 
	\begin{equation*}
	\lim_{n \to \infty} \P^n_x[\tau_y < \tau_x^+] = \P_x[\tau_y \leq \tau_x^+]
	\end{equation*}
	if and only if 
	\begin{equation}
	\label{eq:HittingTimeBoundaryLimitNever}
	\lim_{n \to \infty} \bar{\P^n_x}[\tau_{V_n^*} < \tau_y < \tau_x^+] = \P_x[\tau_x^+ = \tau_y = \infty]
	\end{equation}
	which in turn is equivalent to 
	\begin{equation}
	\label{eq:HittingTimeBoundaryLimitNever2}
	\lim_{n \to \infty} \bar{\P^n_x}[\tau_{V_n^*} < \tau_x^+ < \tau_y] = 0.
	\end{equation}
\end{prop}
\begin{proof}
	Using (\ref{eq:EscapeProbInfiniteRWInTermsOfExtendedRW}) and (\ref{eq:EscapeProbFiniteRWInTermsOfExtendedRW}) from the proof of Proposition \ref{prop:CharacterizationConvergence}, we have
	\[
	\P_x[\tau_x^+ \leq \tau_y] = \P_x[\tau_x^+ = \tau_y = \infty] + \lim_{n \to \infty} \sum_{\gamma \in \Gamma_{\bar{G_n}}(x,y;V_n)} \bar{\P^n_x}(\gamma)
	\]
	and
	\[
	\lim_{n \to \infty} \P^n_x[\tau_x^+ < \tau_y] = \lim_{n \to \infty} \sum_{\gamma \in \Gamma_{\bar{G_n}}(x,y)} \bar{\P^n_x}(\gamma).
	\]
	Hence, we have convergence as desired if and only if
	\[
	\P_x[\tau_x^+ = \tau_y = \infty] = \lim_{n \to \infty} \sum_{
		\substack{
			\gamma \in \Gamma_{\bar{G_n}}(x,y)\\ 
			\gamma \notin \Gamma_{\bar{G_n}}(x,y;V_n)
	}} \bar{\P^n_x}(\gamma) \ = \lim_{n \to \infty} \bP[\tau_{V_n^*} < \tau_y < \tau_x^+].
	\]
	On the other hand, we have
	\begin{align*}
	\P_x[\tau_x^+ = \tau_y = \infty] & = \lim_{n \to \infty} \P_x[\tau_{V \setminus V_n} < \min(\tau_x^+, \tau_y)]\\
	& = \lim_{n \to \infty} \bP[\tau_{V_n^*} < \min(\tau_x^+, \tau_y)]\\
	& = \lim_{n \to \infty} \left( \bP[\tau_{V_n^*} < \tau_x^+ < \tau_y] + \bP[\tau_{V_n^*} < \tau_y < \tau_x^+] \right)
	\end{align*}
	which implies the second claim.
\end{proof}
\begin{rem}
	An equivalent approach would be to consider a \emph{lazy} random walk on $G_n$ which has the same transition probabilities $p(v,w)$ as $\P_x$ for $v \neq w$ but stays at $v$ with probability 
	\[
		\sum_{w \in V \setminus V_n} p(v,w) = \P_v[\omega_1 \in V \setminus V_n].
	\]
	In that case the notion of "stepping out of $V_n$" would be modeled by staying at any vertex $v \in V_n$.
\end{rem}

\section{Embedding $\mathcal{T}$ into transient graphs}
We will show that whenever a graph $G$ is transient and not part of an infinite line, one can find a subgraph of $G$ which is similar to $\mathcal{T}$ from Section \ref{section:InfiniteT}. We will also show that this is sufficient for (\ref{eq:HittingTimeBoundaryLimitZero}) not to hold.

%\begin{lemma}
%	\label{lemma:TransientWalkLeavesEveryFiniteSet}
%	Let $G$ be a connected, transient graph and $F \subseteq V$ a finite set of vertices. For any $x \in F$, we have
%	$$\P_x[\tau_{\partial_o F} < \infty] = 1$$
%	and there exists $v \in \partial_o F$ such that
%	$$\P_v[\tau_F = \infty] > 0.$$
%\end{lemma}
%\begin{proof}
%	\todo{Is the proof really needed?}
%	First, note that because $G$ is transient and thus infinite, $\partial_o F \neq \emptyset$. By definition of transience, we have
%	$$\P_x[V_x = \infty] = 0$$
%	for every $x \in V$. For $x \in F$, we have
%	\begin{align*}
%	0 & \leq \P_x[\tau_{\partial_o F} = \infty]  \leq \P_x[V_F = \infty] \\
%	& \leq \sum_{y \in F}\P_x[V_y = \infty] = \sum_{y \in F} \P_x[\tau_y < \infty] \cdot \P_y[V_y = \infty] = 0.
%	\end{align*}
%	Hence, $\P_x[\tau_{\partial_o F} = \infty] = \P_x[V_F = \infty] = 0$ for all $x \in F$. It follows that $\P_x[\tau_{\partial_o F} < \infty] = 1$.
%	
%	For the second statement, we compute
%	\begin{align*}
%	1 & = \P_x[V_F<\infty]  = \sum_{T = 0}^{\infty} \P_x[\omega_T \in F, \omega_n \notin F \all n > T]\\
%	& = \sum_{T = 0}^{\infty} \sum_{v \in \partial_o F} \P_x[\omega_T \in F, \omega_n \notin F \all n > T, \omega_{T+1} = v]\\
%	& = \sum_{T = 0}^{\infty} \sum_{v \in \partial_o F} \P_x[\omega_T \in F, \omega_{T+1} = v] \cdot \P_v[\tau_F = \infty].
%	\end{align*}
%	Hence, there exists some $v \in \partial_o F$ such that $\P_v[\tau_F= \infty] > 0$.
%\end{proof}
\begin{prop}
	\label{prop:TransientTImpliesStrongTransience}
	Let $G$ be a transient, connected graph which is not a subgraph of a line. Then, there exist $x,y,z \in V$ such that $x \neq y$, $(x,z,y)$ is a path in $G$ and
	$$\P_z [\tau_x = \tau_y = \infty] > 0.$$
\end{prop}
\begin{proof}
	\begin{figure}
		\centering
		\begin{tikzpicture}[scale=1.2,
		vertex/.style={draw, circle, fill=black!10, minimum size=5mm, inner sep=0.2mm, align=center},
		faded/.style={draw=black!60, dashed}
		]
		
		\draw [rounded corners,fill=blue,opacity=0.1] 
		(-0.5,0.5) -- 
		(-0.5,1.5) -- 
		(0.5,1.5) --
		(1.5,0.5) --
		(1.5,-0.5) --
		(0.5,-1.5) --
		(-0.5,-1.5) --
		(-0.5,-0.5) --
		(0,-0.5) arc (-90:90:0.5) --
		cycle;
		
		\node [color=blue](F) at (0.8,0.6) {$F$};
		
		\node [vertex] (Z) at (0,0) {$z$};
		\node [vertex] (X) at (0,1) {$x$};
		\node [vertex] (Y) at (0,-1) {$y$};
		\node [vertex] (V2) at (1,0) {$v'$};
		\node [vertex] (V) at (2,0) {$v$};
		
		\draw (Z) -- (X);
		\draw (Z) -- (Y);
		\draw (Z) -- (V2);
		\draw (V2) -- (V);
		
		\draw [faded] (Z) -- (-1.2,0.8);
		\draw [faded] (Z) -- (-1.5,0);
		\draw [faded] (Z) -- (-1.2,-0.8);
		
		%\draw [faded] (V) -- (3,1);
		\draw [faded] (V) -- +(0.9cm,0);
		%\draw [faded] (V) -- (3,-1);
		\draw [faded] (X) -- +(0,0.7cm);
		\draw [faded] (Y) -- +(0,-0.7cm);
		
		\end{tikzpicture}
		\caption{\label{fig:ProofVisualization}}
	\end{figure}
	Since $G$ is transient, it is infinite. If $G$ is not a subgraph of a line, then there exists some $z \in V$ with at least three adjacent vertices. Let $F$ be a set of exactly three neighbors of $z$. Since $G$ is transient and $F$ is finite, there exists $v \in \partial_o F$ such that 
	\[
	\P_v[\tau_{F} = \infty] > 0.
	\]
	If $v = z$, we can choose $x,y \in F$, $x \neq y$, and get 
	\[
	\P_z[\tau_x = \tau_y = \infty] \geq \P_v[\tau_F = \infty] > 0.
	\]
	If $v \neq z$, then there exists $v' \in F$ such that $(z,v',v)$ is a path in $G$. Let $x,y \in V$ such that $F = \set{x,y,v'}$, see Figure \ref{fig:ProofVisualization}. It follows that
	\begin{align*}
	\P_z[\tau_x = \tau_y = \infty] & \geq \P_z[\omega_1 = v', \omega_2 = v, \tau_x = \tau_y = \infty] \\
	& = p(z,v')\cdot p(v',v) \cdot \P_v[\tau_x = \tau_y = \infty]\\
	& \geq p(z,v')\cdot p(v',v) \cdot \P_v[\tau_F = \infty] > 0.
	\end{align*}
%	In any case, we compute
%	\begin{align*}
%	\P_x[\tau_y = \tau_x^+ = \infty] & \geq \P_x[\tau_y =\tau_x^+ = \infty, \omega_1 = z] = p(x,z) \cdot \P_z[\tau_x = \tau_y = \infty] > 0.
%	\end{align*}
%	and analogously, $\P_y[\tau_x = \tau_y^+ = \infty] > 0$.
\end{proof}
\begin{thm}
	\label{thm:MainResult}
	Let $G$ be a transient, connected graph. Then,
	\[
	\all x,y \in V: \lim_{n \to \infty} \bP[\tau_{V_n^*} < \tau_y < \tau_x^+] = 0
	\]
	holds if and only if $G$ is a subgraph of an infinite line.
\end{thm}

\begin{proof}
First, assume that $G$ is a subgraph of an infinite line and let $x,y \in V$, $x \neq y$. Then, for any $n \in \N$ sufficiently big, we have
\[
\Gamma_{\bar{G_n}}(x,y) \setminus \Gamma_{\bar{G_n}}(x,y;V_n) = \emptyset,
\]
i.e. there exists no path $x \to y$ which leaves $V_n$ before reaching $y$. Hence,
\[
\lim_{n \to \infty} \bP[\tau_{V_n^*} < \tau_y < \tau_x^+] = 0.
\]

To prove the converse direction, suppose that $G$ is not a subgraph of a line. By Proposition \ref{prop:TransientTImpliesStrongTransience}, we know that there exist distinct vertices $x,y,z \in V$ such that $(x,z,y)$ is a path in $G$ and $\P_z[\tau_x = \tau_y = \infty] > 0$. Hence,
\begin{align*}
0 & < \P_z[\tau_x = \tau_y = \infty] \\
& = \lim_{n \to \infty} \P_z[\tau_{\partial_o V_n} < \min(\tau_x, \tau_y)]\\
& = \lim_{n \to \infty} \bar{\P_z^n}[\tau_{V_n^*} < \min(\tau_x,\tau_y)]\\
& = \lim_{n \to \infty} \left( \bar{\P_z^n}[\tau_{V_n^*} < \tau_x < \tau_y] + \bar{\P_z^n}[\tau_{V_n^*} < \tau_y < \tau_x]\right)\\
& \leq \limsup_n \bar{\P_z^n}[\tau_{V_n^*} < \tau_x < \tau_y] + \limsup_n \bar{\P_z^n}[\tau_{V_n^*} < \tau_y < \tau_x].
\end{align*}
Without loss of generality assume that $\limsup_{n \to \infty} \bar{\P_z^n}[\tau_{V_n^*} < \tau_y < \tau_x] > 0$. It follows that $\limsup_{n \to \infty}\bar{\P_x^n}[\tau_{V_n^*} < \tau_y < \tau_x^+] > 0$ because for all $n \in \N$, we have
\begin{align*}
\bar{\P_x^n}[\tau_{V_n^*} < \tau_y < \tau_x^+] & \geq \bar{\P_x^n}[\tau_z < \tau_{V_n^*} < \tau_y < \tau_x^+]\\
& = \bar{\P_x^n}[\tau_z < \min(\tau_{V_n^*}, \tau_x^+, \tau_y)] \cdot \bar{\P_z^n}[\tau_{V_n^*} < \tau_y < \tau_x^+]\\
& \geq p(x,z) \cdot \bar{\P_z^n}[\tau_{V_n^*} < \tau_y < \tau_x].
\end{align*}
\end{proof}

\begin{cor}
	\label{cor:MainResult}
	Let $G$ be a transient, connected graph. Then,
	\begin{equation*}
	R^F(x,y) = \frac{1}{c_x \cdot \P_x[\tau_y < \tau_x^+]}
	\end{equation*}
	holds for all $x,y \in V$ if and only if $G$ is a subgraph of an infinite line.
\end{cor}
\begin{proof}
	As seen in (\ref{eq:ConvergenceOfProbabilites}), the desired probabilistic representation (\ref{eq:FreeEffResProbRepr}) holds if and only if 
	\[
	\lim_{n \to \infty} \P_x^n[\tau_y < \tau_x^+] = \P_x[\tau_y < \tau_x^+].
	\]
	By Proposition \ref{prop:CharacterizationConvergence}, this is equivalent to 
	\[
	\lim_{n \to \infty} \bar{\P^n_x}[\tau_{V_n^*} < \tau_y < \tau_x^+] = 0
	\]
	and the claim follows by Theorem \ref{thm:MainResult}.
\end{proof}

\begin{thm}
	\label{thm:ConvergenceLowerBoundImpliesRecurrence}
	Let $G$ be an infinite graph. If
	\[
	\all x,y \in V: \lim_{n \to \infty} \bP[\tau_{V_n^*} < \tau_y < \tau_x^+] = \P_x[\tau_x^+ = \tau_y = \infty]
	\]
	holds, then $G$ is recurrent.
\end{thm}
\begin{proof}
	By Proposition \ref{prop:CharacterizationConvergenceLowerBound}, we have
	\begin{equation}
	\label{eq:HittingTimeLimitLowerBound}
	\all x,y \in V: \lim_{n \to \infty} \P_x[\tau_{V_n^*} < \tau_x^+ < \tau_y] = 0.
	\end{equation}
	
	Suppose that $G$ is transient and not a subgraph of a line. Using the same arguments as in the proof of Theorem \ref{thm:MainResult}, we see that there exist distinct vertices $x,y,z \in V$ such that $(x,z,y) \in \Gamma_G(x,y)$ and
	\[
	\limsup_n \bar{\P_z^n}[\tau_{V_n^*} < \tau_x < \tau_y] + \limsup_n \bar{\P_z^n}[\tau_{V_n^*} < \tau_y < \tau_x] > 0.
	\]
	Since	
	\[
	\bP[\tau_{V_n^*} < \tau_x^+ < \tau_y] \geq p(x,z) \cdot \bar{\P^n_z}[\tau_{V_n^*} < \tau_x < \tau_y],
	\]
	it follows that 
	\[
	\limsup_n \bar{\P_z^n}[\tau_{V_n^*} < \tau_x < \tau_y] = 0
	\]
	which implies
	\[
	\limsup_{n \to \infty} \bar{\P^n_z}[\tau_{V_n^*} < \tau_y < \tau_x] > 0.
	\]
	However, we also have
	\[
	\bar{\P^n_y}[\tau_{V_n^*} < \tau_y^+ < \tau_x] \geq p(y,z) \cdot \bar{\P^n_z}[\tau_{V_n^*} < \tau_y < \tau_x]
	\]
	and it follows that 
	\[
	\limsup_n \bar{\P_y^n}[\tau_{V_n^*} < \tau_y^+ < \tau_x] > 0
	\]
	which is a contradiction to (\ref{eq:HittingTimeLimitLowerBound}).
	
	Hence, if $G$ is transient, then it must be a subgraph of a line. In this case,
	\[
		\lim_{n \to \infty} \bP[\tau_{V_n^*} < \tau_y < \tau_x^+] = 0
	\]
	follows for all $x,y \in V$ by Theorem \ref{thm:MainResult}. Together with (\ref{eq:HittingTimeLimitLowerBound}), this implies
	\[
	\P_x[\tau_x^+ = \tau_y = \infty] = 0
	\]
	for all $x,y \in V$. However, this is a contradiction to the transience of $G$.
	
\end{proof}
\begin{cor}
	\label{cor:LowerBoundCharacterization}
	Let $G$ be an infinite, connected graph. Then, 
	\begin{equation}
	\label{eq:LowerBound}
	R^F(x,y) = \frac{1}{c_x \cdot \P_x[\tau_y \leq \tau_x^+]}
	\end{equation}
	holds for all $x,y \in V$ if and only if $G$ is recurrent.
\end{cor}
\begin{proof}
	If $G$ is recurrent, we have $\P_x[\tau_x^+ = \tau_y = \infty] = 0$ for all $x,y \in V$. Hence, 
	\[
	\P_x[\tau_x^+ < \tau_y] = \P_x[\tau_x^+ \leq \tau_y]
	\]
	and (\ref{eq:BarlowBounds}) implies the claim.
	
	If (\ref{eq:LowerBound}) holds for all $x,y \in V$, then we have
	\[
		\lim_{n \to \infty} \P^n_x[\tau_y < \tau_x^+] = \P_x[\tau_y \leq \tau_x^+] 
	\]
	for all $x,y \in V$ and Theorem \ref{thm:ConvergenceLowerBoundImpliesRecurrence} implies the recurrence of $G$.
\end{proof}
This shows that the lower bound in (\ref{eq:BarlowBounds}) is actually a strict inequality for transient graphs.

\bibliographystyle{alpha}
\bibliography{refs_graphs}

\end{document}